\documentclass[11pt]{article}\textwidth 160mm\textheight 235mm
\usepackage{amsfonts} \usepackage{amssymb} \usepackage{graphicx}
\usepackage{amsmath} \usepackage{amsthm} \usepackage{mathrsfs}
\usepackage{tikz} \usepackage{cancel} \usepackage{todonotes}
\usetikzlibrary{matrix} \usetikzlibrary{arrows,shapes}
\usepackage{cite} 
\usepackage{hyperref}
\usepackage{amsmath}
\usepackage{amsthm}
\usepackage{amsfonts}
\usepackage{multirow}
\usepackage{amssymb}
\usepackage{graphicx}
\usepackage{tcolorbox}
\usepackage{booktabs}

\newtheorem{Theorem}{Theorem}[section]
\newtheorem{Proposition}[Theorem]{Proposition}
\newtheorem{Remark}[Theorem]{Remark}
\newtheorem{Lemma}[Theorem]{Lemma}
\newtheorem{Corollary}[Theorem]{Corollary}
\newtheorem{Definition}[Theorem]{Definition}
\newtheorem{Example}[Theorem]{Example}

\usepackage{hyperref}
\hypersetup{colorlinks=true,urlcolor=blue}
\expandafter\let\expandafter\oldproof\csname\string\proof\endcsname
\let\oldendproof\endproof
\renewenvironment{proof}[1][\proofname]{
\oldproof[\ttfamily\scshape \bf #1.]
}{\oldendproof}

\def\emp{\emptyset}  
\def\dom{{\rm dom}\,}

   \def\B{\mathbb
B}   
\def\ox{\overline{x}} \def\ov{\overline{y}} 
  
\def\tto{\rightrightarrows} \def\Hat{\widehat}

\def\ra{\rangle}
\def\la{\langle} 

\def\ox{\bar{x}} 
\def\oy{\bar{y}} \def\ov{\bar{v}} 
 \def\ov{\bar{v}}

 \def\inte{\mbox{\rm int}\,}
\def\gph{\mbox{\rm gph}\,} 
 \def\dom{\mbox{\rm dom}\,}

  \def\O{\Omega}
 \def\emp{\emptyset} \def\st{\stackrel}
 \def\lm{\lambda}  
\def\al{\alpha}  \def\N{{\rm I\!N}}
\def\R{{\rm I\!R}}

\numberwithin{equation}{section}

\title{\bf
Local Maximal Monotonicity\\ in Variational Analysis and Optimization}
\author{Pham Duy Khanh\footnote{Department of Mathematics, Ho Chi Minh City University of Education, Ho Chi Minh City, Vietnam. E-mail: pdkhanh182@gmail.com} \quad Vu Vinh Huy Khoa\footnote{Department of Mathematics, Wayne State University, Detroit, Michigan, USA. E-mail: khoavu@wayne.edu. Research of this author was partly supported by the US National Science Foundation under grant DMS-2204519.}\quad Boris S. Mordukhovich\footnote{Department of Mathematics, Wayne State University, Detroit, Michigan, USA. E-mail: aa1086@wayne.edu. Research of this author was partly supported by the US National Science Foundation under grants DMS-1808978 and DMS-2204519, by the Australian Research Council under Discovery Project DP-190100555, and by Project~111 of China under grant D21024.}\quad Vo Thanh Phat\footnote{Department of Mathematics, Wayne State University, Detroit, Michigan, USA. E-mail: phatvt@wayne.edu. Research of this author was partly supported by the US National Science Foundation under grants DMS-1808978 and DMS-2204519.}}
\begin{document}
\maketitle
\vspace*{-0.1in}
\noindent
{\small{\bf Abstract}. The paper is devoted to a systematic study and characterizations of notions of local maximal monotonicity and their strong counterparts for set-valued operators that appear in variational analysis, optimization, and their applications. We obtain novel resolvent characterizations of these notions together with efficient conditions for their preservation under summation in broad infinite-dimensional settings. Further characterizations of these notions are derived by using generalized differentiation of variational analysis in the framework of Hilbert spaces.
}\vspace*{-0.1in}

\section{Introduction}\label{intro}\vspace*{-0.05in}

It is difficult to overstate the importance of monotonicity and especially {\em $($global$)$ maximal monotonicity} notions and results for set-valued mappings/operators/multifunctions with applications to various areas of operations research, optimization, numerical analysis, systems control, etc. The theory of monotone operators has started in the early 1960s by Minty who obtained in \cite{Minty62} his seminal theorem characterizing maximal monotonicity of set-valued operators in Hilbert spaces via the single-valuedness and surjectivity of resolvents. An extension of Minty's theorem to operators in reflexive Banach spaces and many other fundamental results on global maximal monotonicity were established by Rockafellar (see, e.g., \cite{Rockafellar70,Rockafellar70-Pac}) who applied them, in particular, to developing the {\em proximal point method} (PPM) in numerical optimization. The reader is referred to the books by Bauschke and Combettes \cite{Bauschke2011} and by Simons \cite{Simons08}, which largely reflect the current stage of the theory of monotone operators and their important applications in Hilbert and general Banach spaces, respectively. 

There are two different motivations and formalizations of {\em local maximal monotonicity}, which came from the requirements of variational analysis and optimization. The first notion goes back to Poliquin and Rockafellar \cite{Poli} who used it for the study of {\em tilt stability} of local minimizers, the notion which by now plays a highly important role in both theoretical and numerical aspects of optimization; see \cite{kmp22convex,BorisKhanhPhat,kmptmp,ms21,rtr251,r22} and the references therein for the most recent developments and applications. The second notion of local maximal monotonicity was introduced and further developed by Pennanen \cite{Pen02,Pen03} motivated by applications to local convergence of PPM and related {\em numerical algorithms} of optimization. To the best of our knowledge, both notions of local maximal monotonicity have been studied and applied only in finite dimensions.

In this paper, we address the above notions of local maximal monotonicity for set-valued mappings defined on general {\em Banach spaces} and prove that they are {\em equivalent} if the space in question is {\em reflexive}. Note that all the obtained results are {\em new in finite dimensions}. 

Our major theorem provides an appropriate resolvent characterization of these equivalent notions in reflexive Banach spaces that can be viewed as a {\em local counterpart} of Minty's surjectivity theorem applied to {\em graphical localizations}. We also study and characterize the corresponding {\em strong versions} of local maximal monotonicity. The obtained characterizations allow us to establish verifiable conditions ensuring the {\em preservation} of local maximal monotonicity of set-valued mappings under {\em summation} in reflexive Banach spaces. Our further characterizations of local maximal monotonicity and its strong counterpart are established for set-valued mappings defined on {\em Hilbert spaces}, where we employ powerful tools of {\em variational analysis} and {\em generalized differentiation} that revolve around {\em coderivatives} of multifunctions. Coderivatives have been already used by Mordukhovich and Nghia \cite{MordukhovichNghia1} to characterize {\em strong} local maximal monotonicity and by Chieu et al.\ \cite{CBN} to characterize {\em global} maximal monotonicity of set-valued mappings in Hilbert and finite-dimensional spaces. Now we are able to derive {\em coderivative characterizations} of {\em local maximal monotonicity} (not just its strong version) for arbitrary set-valued operators in Hilbert spaces by using the newly established resolvent characterization of this notion married to the advanced machinery of generalized differentiation. Among strong advantages of the obtained coderivative characterizations are extensive {\em calculus rules} available for coderivatives in the books \cite{Mordukhovich06,Mordukhovich18,Rockafellar98} in both finite and infinite dimensions, which allow us to deal with various structured problems and open the door for subsequent developments and applications. This is done in our forthcoming papers.\vspace*{0.03in}

The rest of the paper is organized as follows. Section~\ref{sec:global-Geometry} recalls and discusses the definitions of {\em globally monotone} and globally {\em maximal monotone} set-valued mappings in arbitrary Banach spaces together with their {\em strong} counterparts. We present here important monotonicity properties needed in what follows, including some new observations and proofs. 

Section~\ref{sec:localmonoproperties} is devoted to the study of two versions of {\em local maximal monotonicity} for set-valued mappings in Banach spaces. The main result here is a {\em resolvent characterization} of both versions of local maximal monotonicity in reflexive Banach spaces, which yields the equivalence of these notions in such a framework. We also prove in this section that local maximal monotonicity is {\em preserved} under the operator {\em summation} in reflexive spaces provided the fulfillment of a certain {\em inner semicontinuity} property, which is shown to be essential for the preservation. 

Section~\ref{sec:strong} concerns {\em local strong monotonicity} of multifunctions in Banach spaces admitting a G\^ateaux smooth renorming. Among other results, we reveal here geometric properties of the space in question ensuring the equivalence between two versions of local strong maximal monotonicity of set-valued operators with a {\em prescribed modulus}. 

Section~\ref{sec:gen-diff} briefly discusses {\em generalized differential tools} of variational analysis in Hilbert spaces that are needed for subsequent characterizations of local maximal and strong maximal monotonicity. Reviewing the major {\em coderivative}  constructions for set-valued mappings employed in such characterizations, we present some of their properties used in what follows.

The main result of Section~\ref{sec:coderi-localmax} provides two equivalent characterizations of local maximal monotonicity of set-valued operators in Hilbert spaces in terms of the {\em positive-semidefiniteness} of appropriate {\em coderivatives} combined with the property of local {\em hypomonotonicity}, which is well understood in variational analysis. The concluding Section~\ref{sec:conc} summarizes the major contributions of the paper and discusses some topics for future research.\vspace*{0.03in}

Finally in this section, we recall some notation used throughout the paper. Unless otherwise stated, all the spaces under consideration are {\em Banach}, and we use the {\em sum norm}
\begin{eqnarray*}
\|(x,y)\|:=\|x\|+\|y\|\;\mbox{ for all }\;x\in X,\;y\in Y
\end{eqnarray*}
on the product spaces $X\times Y$. The symbols $x_k\to x$ and $x_k\st{w}{\to}x$ indicate the strong/norm and weak convergence on $X$. Given a set-valued mapping $F\colon X\tto Y$, denote by
\begin{eqnarray*}
\dom F:=\big\{x\in X\;\big|\;F(x)\ne\emp\big\}\;\mbox{ and }\;\gph F:=\big\{(x,y)\in X\times Y\;\big|\;y\in F(x)\big\}
\end{eqnarray*}
its {\em domain} and {\em graph}, respectively. For a nonempty set $\O\subset X$, we use the notation $x\st{\O}{\to}\ox$ meaning that $x\to\ox$ with $x\in\O$. As always, $\N:=\{1,2,\ldots\}$.
\vspace*{-0.1in}

\section{Global Monotonicity and Maximality in Banach Spaces}\label{sec:global-Geometry}\vspace*{-0.05in}

Let $T\colon X\tto X^*$ be a set-valued mapping acting between a Banach space $X$ and its topological dual $X^*$. Recall that the operator $T$ is (globally) {\em monotone} if 
\begin{eqnarray*}
\la x_1^*- x_2^*, x_1- x_2 \ra \ge 0 \;\text{ whenever }\; (x_1,x_1^*),(x_2,x_2^*)\in \gph T.
\end{eqnarray*}
A monotone operator $T$ is {\em maximal monotone} if $\gph T = \gph S$ for any monotone operator $S:X\rightrightarrows X^*$ with $\gph T \subset \gph S$.

To proceed with strong counterparts of global monotonicity, we first define the {\em duality mapping} $J:X\rightrightarrows X^*$ between $(X,\|\cdot\|)$ and $(X^*,\|\cdot\|_*)$ given by
\begin{equation}\label{duality}
J(x):=\big\{x^* \in X^*\;\big|\;\la x^*,x\ra =\|x\|^2 =\|x^*\|_*^2\big\},\quad x\in X.
\end{equation}
Since $J$ is the subgradient mapping of the convex quadratic function $\frac{1}{2}\|\cdot\|^2$, it is maximal monotone on $X$ with $\dom J=X$. It is well known that the {\em single-valuedness} of the duality mapping $J$ is equivalent to the space $(X,\|\cdot\|)$ being {\em $($G\^ateaux$)$ smooth}, i.e., its norm $\|\cdot\|$ is G\^ateaux differentiability on $X\setminus \{0\}$. The latter class is sufficiently broad including, in particular, all separable and all reflexive Banach spaces; see, e.g., \cite{Ciora,fabian}.

Having \eqref{duality} in hand, we say that $T\colon X\tto X^*$ is (globally) \textit{strongly monotone} on $X$ with modulus $\sigma>0$ (or $\sigma$-strongly monotone) if the shift $T-\sigma J$ is monotone on $X$, i.e., 
\begin{eqnarray*}
\la x_1^* - x_2^*, x_1 - x_2 \ra \ge \sigma \la j(x_1)-j(x_2),x_1-x_2\ra 
\end{eqnarray*}
whenever $(x_1,x_1^*),(x_2,x_2^*)\in \gph T \text{ and } (x_1,j(x_1)),(x_2,j(x_2))\in \gph J$. A $\sigma$-strongly monotone mapping $T$ is \textit{$\sigma$-strongly maximal monotone} if $\gph T = \gph S$ for any $\sigma$-strongly monotone operator $S:X\rightrightarrows X^*$ with $\gph T \subset \gph S$.\vspace*{0.03in}

The following  simple observation is useful in what follows.\vspace*{-0.05in}

\begin{Proposition}\label{prop:STR}
Let $T,S,R:X\rightrightarrows X^*$ be set-valued mappings satisfying the inclusion
\begin{eqnarray}\label{eq:TSR}
\gph (T-S) \subset \gph R. 
\end{eqnarray}
Then we have the equivalent properties:

{\bf(i)} $\gph T \subset \gph (R+S)$.

{\bf (ii)} $\dom T \subset \dom S$.
\end{Proposition}
\begin{proof}
To verify first that (i)$\Longrightarrow$(ii), we deduce from  $\dom (R+S) = \dom R \cap \dom S$ and $\gph T  \subset \gph (R+S)$ that $\dom T \subset \dom R \cap \dom S \subset \dom S$, which gives us (ii). To justify the reverse implication, pick $(x,x^*)\in \gph T$ giving us $x\in \dom T$ and hence $x\in \dom S$. This allows us to find $y^*\in S(x)$ such that $(x,x^*-y^*) \in \gph (T-S)$. It follows from \eqref{eq:TSR} that $(x,x^*-y^*)\in \gph R$, which yields $x^*-y^* \in R(x)$ and $x^*\in R(x)+y^* \subset (R+S)(x)$. Therefore, we get $(x,x^*) \in \gph (R+S)$, which verifies $\gph T \subset \gph (R+S)$ and thus completes the proof.
\end{proof}\vspace*{-0.05in}

The next proposition presents some results about the {\em preservation of maximality} of set-valued mappings under their addition/subtraction with the duality mapping $J$. \vspace*{-0.05in}

\begin{Proposition}\label{prop:preser}
Let $T:X\rightrightarrows X^*$, and let $J:X\rightrightarrows X^*$ be the duality mapping \eqref{duality}. For any fixed $\sigma>0$, we have the following statements:

{\bf(i)} If $T$ is $\sigma$-strongly maximal monotone, then $T-\sigma J$ is maximal monotone on $X$.

{\bf(ii)} If $X$ is reflexive and $T$ is maximal monotone, then the shifted operator $T+\sigma J$ is also maximal monotone. 

{\bf(iii)} If $T$ is maximal monotone with the
closed and convex domain $\dom T$, then $T+\sigma J$ is maximal monotone.

{\bf(iv)} If in addition to the assumptions in either {\rm(ii)} or {\rm(iii)}, the norm $\|\cdot\|$ is G\^ateaux differentiable on $X\setminus\{0\}$, then $T+\sigma J$ is $\sigma$-strongly maximal monotone.
\end{Proposition}
\begin{proof} To verify (i), assume that $\gph (T-\sigma J) \subset \gph S$ for some monotone mapping $S:X\rightrightarrows X^*$. As $\gph (T-\sigma J) \subset \gph S$ and $\dom (\sigma J)= X$, Proposition \ref{prop:STR} implies that $\gph T \subset \gph (S+\sigma J)$. The maximality of $T$ yields the equality $\gph T = \gph (S+\sigma J)$. It follows from Proposition~\ref{prop:STR} that $\gph S \subset \gph (T-\sigma J)$, and hence the equality $\gph S = \gph (T-\sigma J)$ holds. Therefore, $T-\sigma J$ is maximal monotone.  

Next we justify assertion (ii). Since $T$ is maximal monotone, we get $\dom T \neq \emptyset$. Moreover, $\sigma J$ is maximal monotone with $\inte(\dom \sigma J) = X$. Thus $\dom T \cap \inte (\dom \sigma J)\neq \emptyset$, and it follows from \cite[Theorem~1]{Rockafellar70} that $T+\sigma J$ is maximal monotone under the reflexivity of $X$. 

For (iii), we have that the set $\dom T \cap \dom (\sigma J) = \dom T$ is closed, that the set $\mathrm{cl}(\dom T) = \dom T$ is convex, and that $\dom T \cap \mathrm{int}(\dom \sigma J) = \dom T \neq \emptyset$. All of the above ensure the maximal monotonicity of $T+\sigma J$ by using the result of \cite[Theorem~5.10($\eta$)]{Voisei08}.

If for (iv) the norm $\|\cdot\|$ is G\^ateaux smooth on $X\setminus\{0\}$ in (ii) and (iii), then $J$ is single-valued on the whole space $X$, which ensures that $T+\sigma J$ is $\sigma$-strongly monotone by the monotonicity of $T=(T+\sigma J)-\sigma J$. Combining  this with the maximality justified above, we arrive at the $\sigma$-strong maximal monotonicity of $T+\sigma J$ and thus complete the proof.
\end{proof}\vspace*{-0.05in}

In fact, the assumptions of Proposition~\ref{prop:preser} allow us to get the following necessary and sufficient condition for the $\sigma$-strong monotonicity of set-valued mappings.\vspace*{-0.05in}

\begin{Proposition}\label{strong-alter-glo}
Let $(X,\|\cdot\|)$ be a G\^ateaux smooth Banach space, and let $T:X\tto X^*$ be a $\sigma$-strongly monotone operator for some $\sigma >0$. Assume that either $X$ is reflexive, or the set $\dom T$ is closed and convex. Then $T$ is $\sigma$-strongly maximal monotone if and only if 
\begin{equation}\label{strong-alter}
\gph T = \gph S \text{ for any $\sigma$-strongly monotone }S:X\tto X^* \text{ with }\gph T \subset \gph S.
\end{equation}
\end{Proposition}
\begin{proof}
It is clear that the $\sigma$-strong maximal monotonicity of $T$ yields \eqref{strong-alter}. Conversely, suppose that $T$ satisfies \eqref{strong-alter} and show that $T$ is $\sigma$-strongly maximal monotone. First check that the $\sigma$-strong monotonicity of $T$ and \eqref{strong-alter} imply that $T-\sigma J$ is maximal monotone. Indeed, taking into account that  $T-\sigma J$ is monotone on $X$ and considering any monotone mapping $S:X\tto X^*$ with $\gph (T-\sigma J)\subset \gph S$ tell us by Proposition~\ref{prop:STR} that $\gph T \subset \gph (S+\sigma J)$. Since the operator $S+\sigma J$ is $\sigma$-strongly monotone on $X$, we deduce from \eqref{strong-alter} that $\gph T = \gph (S+\sigma J)$, and hence $\gph (T-\sigma J)= \gph S$. This verifies the maximal monotonicity of $T-\sigma J$. To complete the proof of the proposition, observe that the assumptions on either the reflexivity of $X$ or on the closedness and convexity of $\dom T$ together with the G\^ateaux smoothness of $\|\cdot\|$ ensure by Proposition~\ref{prop:preser} that the operator $T=(T-\sigma J)+\sigma J$ is $\sigma$-strongly maximal monotone.
\end{proof}\vspace*{-0.05in}

We conclude this section with the following proposition, which shows that the study of monotone (resp.\ strongly monotone) mappings can be largely reduced to that of their maximal (resp.\ strongly maximal monotone) versions.\vspace*{-0.05in} 

\begin{Proposition}\label{lem:extend}
Let $T:X\rightrightarrows X^*$, where $(X,\|\cdot\|)$ is an arbitrary Banach space. We have:

{\bf(i)} If $T$ is monotone, then there exists a maximal monotone operator $\bar{T}:X\rightrightarrows X^*$ extending $T$ in the sense that $\gph T \subset \gph \bar{T}$.

{\bf(ii)} If $X$ is reflexive and G\^ateaux smooth, and if $T$ is $\sigma$-strongly monotone with some $\sigma>0$, then there exists a $\sigma$-strongly maximal monotone operator $\bar{T}:X\rightrightarrows X^*$ extending $T$.
\end{Proposition}
\begin{proof} Assertion (i) is proved by employing Zorn's lemma; see, e.g., \cite[Proposition~12.6]{Rockafellar98}. Let us now derive (ii) from (i) by the usage of the reflexivity of $X$ and the G\^ateaux smoothness of the norm $\|\cdot\|$. If $T$ is $\sigma$-strongly monotone with some $\sigma>0$, then the shift $T-\sigma J$ is monotone, and thus we get by (i) a maximal monotone mapping $\bar{T}:X\rightrightarrows X^*$ with $\gph (T-\sigma J) \subset \gph \bar{T}$. From Proposition~\ref{prop:STR} and $\dom (\sigma J) = \dom J = X$, it follows that $\gph T \subset \gph (\bar{T}+\sigma J)$. Since the extension $\bar{T}$ is maximal monotone and since the space $(X,\|\cdot\|)$ is reflexive and G\^ateaux smooth, Proposition~\ref{prop:preser}(ii) ensures that $\bar{T}+\sigma J$ is $\sigma$-strongly maximal monotone. This shows that the operator $\bar T$ satisfies all the requirements in (ii), which thus completes the proof.
\end{proof}\vspace*{-0.2in}

\section{Local Monotonicity of Set-Valued Mappings}\label{sec:localmonoproperties}\vspace*{-0.05in}

This section addresses {\em local} monotonicity of set-valued mappings defined  on Banach spaces and study the two versions of {\em local maximal} monotonicity discussed in Section~\ref{intro}. The proof of the main result here uses some notions and facts from {\em geometry theory} of Banach spaces, which we overview first based on the books \cite{Ciora,fabian,nam}, where the reader can find more details. 

Similarly to the case of G\^ateaux smoothness, the norm $\|\cdot\|$ on $X$ is {\em Fr\'echet smooth} if it is Fr\'echet differentiable on $X\setminus \{0\}$. This is equivalent to saying that the duality mapping \eqref{duality} is single-valued and continuous on $X$. Note that in this case, the function $\frac{1}{2}\|\cdot\|^2$ is continuously differentiable on $X$ with $\nabla \left( \frac{1}{2}\|\cdot\|^2\right)=J$. It is well known that if either $X$ is reflexive or $X^*$ is separable, then $(X,\|\cdot\|)$ admits a Fr\'echet smooth {\em renorming}, i.e., an equivalent norm that is Fr\'echet differentiable at all nonzero points. 

Further, a Banach space $(X,\|\cdot\|)$ is {\em strictly convex} if for all $x,y\in S_X:=\{x\in X\mid \|x\|=1\}$, $x\neq y$, we have
$\|\lambda x + (1-\lambda)y\|<1$ for all $\lambda\in (0,1)$. Recall finally that $(X,\|\cdot\|)$ is \textit{Kadec-Klee} if whenever $x_k\xrightarrow{w}x$ and $\|x_k\|\rightarrow \|x\|$,
the strong convergence $x_k \rightarrow x$ as $k\to\infty$ is guaranteed. The fundamental achievements in geometric theory of Banach spaces tell us that in {\em any reflexive} space $X$, there exists an equivalent norm $\|\cdot\|$ such that $(X,\|\cdot\|)$ is
\begin{equation}\label{Troyanski-Asplund}
\text{Fr\'echet smooth, strictly convex, and Kadec-Klee.}
\end{equation}

Now we are ready to proceed with the study of locally monotone multifunctions with paying the main attention to the maximality issues. The notions of local monotonicity is defined in the following natural way as a localization of global monotonicity. A set-valued mapping $T\colon X\tto X^*$ on a Banach space $X$ is {\em locally monotone} around a point $(\ox,\ox^*)\in \gph T$ if there is an (open) neighborhood $U\times V$ of $(\ox,\ox^*)$ and a globally monotone operator $\bar{T}:X\rightrightarrows X^*$ such that 
\begin{eqnarray}\label{eq:local-mono}
\gph \bar{T} \cap (U\times V) = \gph T \cap (U\times V).
\end{eqnarray}
The mapping $T$ is called {\em monotone with respect to} $W\subset X\times X^*$ ($W$ is not necessarily open) if \eqref{eq:local-mono} holds with replacing $U\times V$ by $W$.\vspace*{0.03in}

For local {\em maximal} monotonicity, we have the following two versions labeled as types $(A)$ and $(B)$, which were originated in \cite{Poli} and \cite{Pen02}, respectively.\vspace*{-0.03in}

\begin{Definition}\label{defi:local-max-R} 
Let $T:X\rightrightarrows X^*$, and let $(\ox,\ox^*)\in\gph T$, where $X$ is a Banach space. Then:

{\bf(i)} $T$ is {\sc locally maximal monotone of type (A)} around $(\ox,\ox^*)$ if there is a neighborhood $U\times V$ of $(\ox,\ox^*)$ such that $T$ is monotone with respect to $U\times V$ and that $\gph T\cap (U\times V)=\gph S\cap (U\times V)$ for any mapping $S:X\rightrightarrows X^*$, which is monotone with respect to $U\times V$ and satisfies the inclusion $\gph T\cap (U\times V)\subset \gph S \cap (U\times V)$.
    
{\bf(ii)} $T$ is {\sc locally maximal monotone of type (B)} around $(\ox,\ox^*)$ if there are a neighborhood $U\times V$ of $(\ox,\ox^*)$ and a maximal monotone mapping $\bar{T}$ such that \eqref{eq:local-mono} holds.
\end{Definition}\vspace*{-0.05in}

Similarly to the above, we understand the monotonicity of types (A) and (B) {\em with respect to some $W\subset X\times X^*$}. The following proposition easily follows from the definitions.\vspace*{-0.03in} 

\begin{Proposition}\label{propo:R->P}
If $T:X\rightrightarrows X^*$ is locally maximal monotone of type $(A)$ with respect to $W\subset X\times X^*$, then this mapping is locally maximal monotone of type $(B)$ with respect to $W$.
\end{Proposition}
\begin{proof}
Assuming that $T$ is locally maximal monotone of type $(A)$ with respect to $W$, we find a globally maximal monotone mapping $\bar T\colon X\tto X^*$ with $\gph \bar{T} \cap W = \gph T \cap W$. Let $F:X\rightrightarrows X^*$ be such that $\gph F := \gph T \cap W$, which yields the global monotonicity of $F$. It follows from Proposition~\ref{lem:extend}(i) that there exists a maximal monotone operator $\bar{T}:X\rightrightarrows X^*$ satisfying $\gph F \subset \gph \bar{T}$. Therefore, we arrive at the inclusion
\begin{equation}\label{216}
\gph T \cap W \subset \gph \bar{T}. 
\end{equation}
The monotonicity of $\bar{T}$ and local monotonicity (A) of $T$ with respect to $W$ imply by \eqref{216} that $\gph T \cap W = \gph \bar{T} \cap W$. This gives us a maximal monotone operator satisfying $\gph T \cap W = \gph \bar{T} \cap W$. Thus $T$ is locally maximal monotone of type $(B)$ with respect to $W$.
\end{proof}\vspace*{-0.05in}

A natural question arises on whether we have the {\em equivalence} between the two types of local maximal monotonicity in Definition~\ref{defi:local-max-R}. In what follows, we answer this question in the {\em affirmative} in the case of {\em reflexive} Banach spaces. In fact, we deduce this equivalence from the {\em resolvent characterization} of both notions in our major theorem given below.\vspace*{0.03in}

To proceed, recall that given a set-valued mapping $T:X\rightrightarrows Y$ and a point $(\ox,\oy)\in\gph T$, it is said that $\widehat{T}:X\rightrightarrows Y$ is a {\em $($graphical$)$ localization} of $T$ around $(\ox,\oy)$ if there is a neighborhood $U\times V$ of this point such that $\gph \widehat{T}= \gph T\cap (U\times V)$. Furthermore, $T$ has a {\em single-valued localization} around $(\ox,\oy)$ if the mapping $\widehat{T}$ is single-valued on $U$ with $\dom \widehat{T} = U$. If in addition $\widehat{T}$ is continuous on $U$ then we say that $T$ admits a {\em continuous single-valued localization} around $(\ox,\oy)$. It is obvious that if $T$ has a (continuous) single-valued localization around $(\ox,\oy)$ and $\sigma \neq 0$, then the mapping $\sigma T$ has a (continuous) single-valued localization around $(\ox,\sigma\oy)$, and also the inverse mapping $T^{-1}$ has a (continuous) single-valued localization around $(\oy,\ox)$.\vspace*{0.03in}

The following crucial result provides a characterization of both notions of local maximal monotonicity of set-valued operators $T\colon X\tto X^*$ in terms of continuous single-valued localizations of their {\em resolvents} $( J+\lambda T)^{-1}$, where $J$ is the duality mapping \eqref{duality} associated with the $\|\cdot\|$ which satisfies the properties in \eqref{Troyanski-Asplund}.
Recall that such an equivalent norm is available in any reflexive Banach space. The obtained characterization can be viewed as a {\em local counterpart} of the celebrated Minty theorem on global monotonicity.\vspace*{-0.05in}

\begin{Theorem}\label{theo:loc-Minty}
Let $X$ be a reflexive space endowed with an equivalent norm $\|\cdot\|$ satisfying the properties in \eqref{Troyanski-Asplund}, and let $J$ be the associated duality mapping \eqref{duality}. Given a multifunction $T:X\rightrightarrows X^*$ and a point $(\ox,\ox^*)\in \gph T$, the following assertions are equivalent:

{\bf(i)} $T$ is locally maximal monotone of type $(B)$ around $(\ox,\ox^*)$.

{\bf(ii)} $T$ is locally monotone around $(\ox,\ox^*)$ and the resolvent $(J+\lambda T)^{-1}$ has a continuous single-valued localization around $( J(\ox)+\lambda \ox^*,\ox)$ for any $\lm>0$.

{\bf(iii)} $T$ is locally maximal monotone of type $(A)$ around $(\ox,\ox^*)$.
\end{Theorem}
\begin{proof} To verify (i) $\Longrightarrow$(ii), let $T$ be locally maximal monotone of type $(B)$ around $(\ox,\ox^*)$ and then find a maximal monotone operator $\bar{T}$ on $X$ and a neighborhood $U\times V$ of $(\ox,\ox^*)$ with 
\begin{equation}\label{T-Tbar}
\gph \bar{T} \cap (U\times V) = \gph T \cap (U\times V).
\end{equation}
Since $J$ is continuous, the mapping $\Psi_{ \lambda}:X\times X^* \rightarrow X^*\times X$ defined by 
\begin{equation*}
\Psi_{\lambda}(x,x^*): =\big(J(x)+\lambda x^*,x\big) \ \text{ for all }\;(x,x^*)\in X\times X^*,\quad \lm>0
\end{equation*}
is a homeomorphism. By taking the image of $\Psi_{\lambda}$ on both sides of \eqref{T-Tbar}, we have
\begin{equation}\label{eq:W^*}
\gph (J+\lambda \bar{T})^{-1} \cap \Psi_{ \lambda}(U\times V) = \gph ( J+\lambda T)^{-1} \cap \Psi_{\lambda}(U\times V).
\end{equation}
The reflexivity of $X$ allows us to apply to $\bar T$ the fundamental Rockafellar's characterization of global maximal monotonicity telling us that  $(J+\lambda \bar{T})^{-1}$ is a single-valued operator from $X^*$ to $X$ which is {\em demicontinuous}, i.e., continuous from the strong topology on $X^*$ to the weak topology of $X$. Let us show that $(J+\lambda \bar{T})^{-1}$ is actually {\em continuous} from the strong topology of $X^*$ to the strong topology of $X$. Indeed, fix $x^*\in X^*$ and consider a sequence $\{x_k^*\}\subset X^*$ strongly converging to $x^*$ as $k\to\infty$. For each $k\in \N$, define
\begin{equation}\label{Chuan1}
x_k:= (J+\lambda \bar{T})^{-1}(x_k^*)\;\mbox{ and }\;x:= (J+\lambda \bar{T})^{-1}(x^*).
\end{equation}
We know by the demicontinuity of $(J+\lambda \bar{T})^{-1}$ that $x_k \xrightarrow{w}x$ as $k\to\infty$. It follows from \eqref{Chuan1} that there exist $y_k^*\in \bar{T}(x_k)$ as $k\in\N$ and $y^*\in \bar{T}(x)$ such that
\begin{equation}\label{Chuan2}
x_k^* = J(x_k) + \lambda y_k^*\;\mbox{ and }\;x^* = J(x) + \lambda y^*.
\end{equation}
Since $x_k^*\rightarrow x^*$ and $x_k \xrightarrow{w}x$, we clearly have the convergence
\begin{equation}\label{230}
\la x_k^*-x^*,x_k-x\ra \rightarrow 0 \;\mbox{ as }\;k\to\infty.
\end{equation}
Using \eqref{Chuan2} allows us to rewrite \eqref{230} as
\begin{equation}\label{Chuan3}
\la J(x_k)-J(x),x_k-x\ra + \lambda \la y_k^* - y^*,x_k-x\ra \rightarrow 0\;\mbox{ as }\;k\to\infty.
\end{equation}
By the global monotonicity of the operators $J$ and $\bar{T}$, both terms on the left-hand side of \eqref{Chuan3} are nonnegative, and hence $ \la J(x_k)-J(x),x_k-x\ra\searrow 0$ as $k\to\infty$. Combining the latter with the well-known estimate 
\begin{equation*}
\la J(x_k)-J(x),x_k-x\ra \ge \big(\|x_k\|-\|x\|\big)^2,\quad k\in\N,
\end{equation*}
gives us $\|x_k\|\rightarrow\|x\|$ as $k\to\infty$. Having $x_k\xrightarrow{w}x$ and $\|x_k\|\rightarrow \|x\|$, we arrive at the strong convergence $x_k\rightarrow x$, since the norm $\|\cdot\|$ under consideration possesses the Kadec-Klee property. This justifies the norm-to-norm continuity of $(J+\lambda\bar{T})^{-1}$.

Remembering that $\Psi_{\lambda}$ is a homeomorphism and that $U\times V$ is a neighborhood of $(\ox,\ox^*)$, we get that the image $\Psi_{\lambda}(U\times V)$ is a neighborhood of $(J(\ox)+\lambda \ox^*,\ox) = \Psi_{\lambda} (\ox,\ox^*)$. This allows us to choose neighborhoods $W$ of $J(\ox)+\lambda \ox^*$ and $Q$ of $\ox$ such that $W\times Q\subset \Psi_{\lambda}(U\times V)$. Since $(J+\lambda \bar{T})^{-1}$ is continuous at $J(\ox)+\lambda \ox^*$ and $Q$ is a neighborhood of $\ox=( J+\lambda \bar{T})^{-1}(J(\ox)+\lambda \ox^*)$, we may shrink $W$ if necessary so that $(J+\lambda \bar{T})^{-1}(W)\subset Q$. Combining the above with \eqref{eq:W^*} brings us to the equality
\begin{equation}\label{Q^*}
\gph (J+\lambda \bar{T})^{-1} \cap (W\times Q) = \gph ( J+\lambda T)^{-1} \cap (W\times Q),
\end{equation}
which ensures that $(J+\lambda T)^{-1}$ has a single-valued localization around $(J(\ox)+\lambda \ox^*,\ox)$. The continuity of this localization on $W$ follows from that of $(J+\lambda \bar{T})^{-1}$, and thus we get (ii).\vspace*{0.03in}

Let us now verify implication (ii)$\Longrightarrow$(iii). Supposing that (ii) holds gives us a neighborhood $V\times U$ of $(J(\ox)+\lambda\ox^*,\ox)$ such that the mapping $S:V \rightrightarrows U$ defined by $\gph S := \gph (J+ \lambda T)^{-1}\cap(V\times U)$ is single-valued on $V$ with $\dom S=V$. We aim at proving that $T$ is locally maximal monotone of type $(A)$ with respect to the neighborhood $\Psi_{\lambda}^{-1}(V\times U)$ of $(\ox,\ox^*)=\Psi_{\lambda}^{-1}(J(\ox)+\lambda\ox^*,\ox)$. Indeed, fix a pair $(x,x^*)\in \Psi_{\lambda}^{-1}(V\times U)$ satisfying the conditions
\begin{equation}\label{eq:max-proof1}
\la x^*-y^*,x-y\ra \ge 0
\ \text{ for all }\ (y,y^*)\in \gph T \cap \Psi_{\lambda}^{-1}(V\times U).
\end{equation}
It follows from $(x,x^*)\in \Psi_{\lambda}^{-1}(V\times U)$ that
$$
(J(x)+\lambda x^*,x)=\Psi_{\lambda}(x,x^*)\in \Psi_{\lambda}\big(\Psi_{\lambda}^{-1}(V\times U)\big) = V\times U,
$$
and hence $x\in U$ and $J(x)+\lambda x^*\in V$. By $\dom S=V$ and $\dom S^{-1} = U$, the vector $y:=S(J(x)+\lambda x^*)$ belongs to $U$. From the definition of $S$, we also have that $y\in (J+\lambda T)^{-1}(J(x)+\lambda x^*)$ and $J(x)+\lambda x^*\in (J+\lambda T)(y)$, which yield
\begin{equation}\label{304}
\lambda^{-1}\big(J(x)+\lambda x^*-J(y)\big) \in T(y).
\end{equation}
Observe also the relationships
\begin{equation}\label{305}
\Psi_{\lambda}\big(y,\lambda^{-1}(J(x)+\lambda x^*-J(y))\big)=(J(x) + \lambda x^*,y) \in V\times U.
\end{equation}
It follows from \eqref{304} and \eqref{305} that $\big(y,\lambda^{-1}(J(x)+\lambda x^*-J(y))\big)\in \gph T \cap \Psi_{\lambda}^{-1}(V\times U)$. Therefore, we deduce from \eqref{eq:max-proof1} that
\begin{equation*}
\lambda^{-1}\la J(y)-J(x),x-y\ra \ge 0
\end{equation*}
giving us $\la J(y)-J(x),y-x\ra =0$ due to the global monotonicity of $J$. By the strict convexity of $\|\cdot\|$, the latter implies that $x=y$. Hence $x^* \in T(x)$, this means that $T$ is locally maximal monotone around $(\ox,\ox^*)$ of type $(A)$, which justifies (iii).  The last implication (iii)$\Longrightarrow$(i) was verified in Proposition~\ref{propo:R->P} for general Banach spaces. This completes the proof of the theorem. 
\end{proof}\vspace*{-0.05in}

Having in mind the equivalence between types (A) and (B) in Theorem~\ref{theo:loc-Minty}, we use in what follows the term ``{\em local maximal monotonicity}" for set-valued mappings on reflexive spaces.\vspace*{0.03in}

Next we address the question about {\em preservation} of local maximal monotonicity under summation of two set-valued operators in reflexive Banach spaces. Rockafellar's result in \cite[Theorem~1]{Rockafellar70} establishes the preservation of {\em global maximal monotonicity} for the sum $T_1+T_2$ of two maximal monotone operators in reflexive spaces under {\em qualifications condition}
\begin{equation}\label{Rocka-qualifi}
\dom T_1 \cap \mathrm{int}(\dom T_2)\ne\emptyset.
\end{equation}

We are not familiar with any result concerning preservation of {\em local maximal monotonicity} under summation. Such a local preservation theorem is derived below in reflexive spaces under a certain inner semicontinuity condition, which is shown to be essential for the fulfillment of the local preservation result. The  needed {\em graphical} inner semicontinuity condition is taken from \cite[Definition~1.63(i)]{Mordukhovich06} being different from the standard inner/lower semicontinuity of multifunctions at domain points; see \cite{Mordukhovich06,Rockafellar98}. We say that $F\colon X\tto Y$ between two Banach spaces is (graphically)  {\em inner semicontinuous} at $(\ox,\oy)\in\gph F$ if
\begin{equation}\label{isc}
\forall x_k\rightarrow \ox \ \ \  \exists\, N\subset \N\;\mbox{ such that }\;\N\setminus N\;\text{ is finite and }\;y_k \overset{N}{\rightarrow} \oy\; \text{ with }\;y_k\in F(x_k).
\end{equation}\vspace*{-0.2in}

\begin{Theorem}\label{theo:preser-P-sum}
Let $X$ be a reflexive Banach space, and let $T_1,T_2:X\rightrightarrows X^*$ be set-valued mappings. Take $\ox\in \mathrm{int}(\dom T_1) \cap \dom T_2$, $\ox^*_1\in T_1(\ox)$, and $\ox^*_2\in T_2(\ox)$. If both mappings $T_1$ and $T_2$ are locally maximal monotone around $(\ox,\ox^*_1)$ and $(\ox,\ox^*_2)$, respectively, then the sum $T:=T_1 +T_2$ is also locally maximal monotone around $(\ox,\ox^*_1+\ox_2^*)$ provided that the mapping $T_1$ is graphically inner semicontinuous at $(\ox,\ox^*_1)$.
\end{Theorem}
\begin{proof} 
By the assumptions made, we find neighborhoods $U\times V_1$ and $U\times V_2$ of $(\ox,\ox^*_1)$ and $(\ox,\ox^*_2)$, respectively, such that
\begin{eqnarray}\label{eq:T1bar}
\gph \bar{T}_1 \cap (U\times V_1) = \gph T_1 \cap (U\times V_1),\quad\gph \bar{T}_2 \cap (U\times V_2) = \gph T_2 \cap (U\times V_2)
\end{eqnarray}
with the globally maximal monotone operators $\bar T_1$ and $\bar T_2$. We have $\ox\in \operatorname{int}(\dom \bar{T}_1)$ since $T_1$ (and hence $\bar{T}_1$) is graphically inner semicontinuous at $(\ox,\ox^*_1)$. It follows from \cite[Theorem~1]{Rockafellar70} discussed above that the sum $\bar{T} := \bar{T}_1 + \bar{T}_2$ is maximal monotone on $X$. Select $r>0$ so that $B_{r}(\ox^*_2)\subset V_2$ for the ball centered at $\ox^*_2$ with radius $r$. Consider neighborhoods $W_1\subset V_1$ of $\ox^*_1$ and $W_2$ of $\ox^*_2$ satisfying the condition
\begin{equation*}
\max\{\mathrm{diam}(W_1),\;\mathrm{diam} (W_2)\} < r/2,
\end{equation*}
where $\mathrm{diam}(\Omega):=\sup\{\|x-y\|\mid x,y\in \Omega\}$ is the diameter of a given set. Due to the inner semicontinuity of $T_1$ and hence of $\bar{T}_1$ at $(\ox,\ox^*_1)$, there is a neighborhood $Q\subset U$ of $\ox$ for which
$$
T_1 (Q) \cup\bar{T}_1 (Q) \subset W_1.
$$
With $W:=W_1+W_2$, observe that $Q\times W$ is a neighborhood of $(\ox,\ox^*_1+\ox_2^*)$. To complete the proof of the theorem, it remains to show that $\gph T \cap (Q\times W) = \gph \bar{T}\cap (Q\times W)$. Indeed, picking $(x,y)\in \gph \bar{T} \cap (Q\times W)$ gives us 
\begin{eqnarray}\label{eq:T}
y = y_1 + y_2\in W\;\mbox{ with }\;y_1 \in \bar{T}_1 (x),\; y_2\in \bar{T}_2 (x),\;x \in Q.
\end{eqnarray}
Since $x\in Q$, we have $y_1\in \bar{T}_1 (x) \subset W_1 \subset V_1$, and hence it follows from \eqref{eq:T1bar} that $(x,y_1)\in \gph T_1$. Using $y \in W=W_1+W_2$ allows us to find $z_1 \in W_1$ with $y - z_1 \in W_2$. Observe that
\begin{eqnarray*}
\left\| y - y_1 - \ox^*_2 \right\| &=& \left\| (y - z_1 - \ox^*_2) + (z_1 - y_1) \right\| \\
&\le& \left\|(y - z_1) - \ox^*_2 \right\| + \left\| z_1 - y_1\right\| \\
&\le& \mathrm{diam}(W_2) + \mathrm{diam}(W_1)< r,
\end{eqnarray*}
and thus $y_2 = y-y_1 \in B_{r}(\ox^*_2) \subset V_2$, which yields $(x,y_2)\in \gph \bar{T}_2 \cap (U\times V_2) = \gph T_2 \cap (U\times V_2)$. Since $(x,y_2)\in \gph T_2$ and $(x,y) \in \gph T$, we arrive at $\gph \bar{T}\cap (Q\times W) \subset \gph T\cap (Q\times W)$. The reverse inclusion can be checked similarly. Therefore,
$\gph T\cap (Q\times W)  = \gph \bar{T}\cap (Q\times W)$, which verifies the local maximal monotonicity of $T$ relative to $Q\times W$.
\end{proof}\vspace*{-0.05in}

To conclude this section, we present an example showing that the inner semicontinuity assumption is {\em essential} for the fulfillment of the preservation result in Theorem~\ref{theo:preser-P-sum}.\vspace*{-0.05in}

\begin{Example}
\rm Consider the two multifunctions $T_1,T_2\colon\R^2 \rightrightarrows \R^2$ generated by the {\em normal cone mappings} to convex closed sets as follows
\begin{equation*}
    T_i(x,y):=N\big((x,y);\O_i\big),\;i=1,2,\;\mbox{ for }\;\O_1:=\big\{(x,y)\;|\;y\ge x^2\big\}\;\mbox{ and }\;\O_2:=\R\times\{0\}.    
\end{equation*}
It is well known from convex analysis that $T_1,T_2$ are maximal monotone operators on $\R^2$. Their sum $T:=T_1+T_2$ is calculated by
\begin{equation*}
T(x,y)=\left\{\begin{array}{ll}
    \{(0,t)\in\R^2\;\big|\;t\in\R\}&\mbox{for }\;(x,y)=(0,0),\\
\emp&\mbox{otherwise}.
\end{array}\right.
\end{equation*}
We can easily see that $T$ is not maximally monotone on  $\R^2$. Indeed, by adding any point $\big((\al,0),(0,0)\big)$ with $\al>0$ to the graph of $T$, we get another monotone operator, which contradicts the maximal monotonicity of $T$. Observe that the qualification condition 
\eqref{Rocka-qualifi} clearly fails in this case since $\inte(\dom T_2)=\emp$. 

Let us now check whether the sum operator $T$ is {\em locally} maximal monotone around the point $(0,0)\in T_1(0,0) \cap T_2(0,0)$. Take any neighborhood $U\times V$ of $(0,0)\in\R^2\times\R^2$ and show that $T$ is {\em not} maximal monotone with respect to $U\times V$, i.e., there exists a monotone operator $S:\R^2 \rightrightarrows \R^2$ with respect to $U\times V$ such that the set $\gph T \cap (U\times V)$ is properly contained in $\gph S \cap (U\times V)$. Indeed, constructing $S$ by
\begin{equation*}
S(x,y):=\begin{cases}
T(x,y)&\mbox{ for }\;x=y=0, \\
\{(0,0)\} &\mbox{ for }\;x\neq 0,\; y=0,\\
\emptyset&\text{otherwise},
\end{cases}
\end{equation*}
we see that this mapping satisfies all the required conditions. Observe that $(0,0)\notin\inte(\dom T_1)\cup\inte(\dom T_2)$, and that neither $T_1$ nor $T_2$ in inner semicontinuous at $(0,0)$. 
\end{Example}\vspace*{-0.2in}

\section{Local Strong Monotonicity in Smooth Banach Spaces}\label{sec:strong}\vspace*{-0.05in} 

This section is devoted to the study of {\em strong local monotonicity} and its {\em maximality} versions for set-valued mappings defined on Banach spaces. \vspace*{-0.03in}

\begin{Definition}\label{str-mon} Let $T\colon X\tto X^*$ be a set-valued mapping on a Banach space, let $(\ox,\ox^*)\in\gph T$, and let $\sigma>0$. Then we say that:

{\bf(i)} $T$ is {\sc locally strongly monotone} around $(\ox,\ox^*)$ with modulus $\sigma$ if there are a neighborhood $U\times V$ of $(\ox,\ox^*)$ and a globally $\sigma$-strongly monotone operator $\bar{T}:X\rightrightarrows X^*$ with
\begin{eqnarray}\label{eq:local-strong}
\gph \bar{T} \cap (U\times V) = \gph T \cap (U\times V).
\end{eqnarray}

{\bf(ii)} $T$ is {\sc locally strongly maximal monotone of type $(A)$} around $(\ox,\ox^*)$ with modulus $\sigma$ if there is a neighborhood $U\times V$ of $(\ox,\ox^*)$ such that $T$ is $\sigma$-strongly monotone with respect to $U\times V$, and such that $\gph T\cap (U\times V) = \gph S\cap (U\times V)$ for any mapping $S:X\rightrightarrows X^*$, which is monotone with respect to $U\times V$ and satisfies the inclusion $\gph T\cap (U\times V)\subset \gph S \cap (U\times V)$.

{\bf(iii)}  $T$ is {\sc locally strongly maximal monotone of type $(B)$} around $(\ox,\ox^*)$ with modulus $\sigma$ if there are a neighborhood $U\times V$ of $(\ox,\ox^*)$ and a $\sigma$-strongly maximal monotone mapping $\bar{T}$ such that \eqref{eq:local-strong} holds.
\end{Definition}\vspace*{-0.05in}

As we see, the strong monotonicity notions involve the duality mapping $J$ from \eqref{duality}, which is {\em single-valued} if and only if the space $(X,\|\cdot\|)$ is {\em G\^ateaux smooth}. For simplicity and definiteness, this property is {\em assumed} without further mentioning in the reminder of this section.\vspace*{0.03in}

For any $\sigma\ne 0$, consider the \textit{vertical shear mapping/vertical transvection} $\Phi_{\sigma}:X\times X^* \rightarrow X\times X^*$ defined by 
\begin{equation}\label{eq:shear}
\Phi_{\sigma}(x,x^*):=\big(x,x^*+\sigma J(x)\big)\ \text{ whenever }\;(x,x^*)\in X\times X^*,
\end{equation}
which is a {\em bijection} (actually, a homeomorphism if $\|\cdot\|$ is Fr\'echet smooth) from $X\times X^*$ to itself. It is easy to check that for any set-valued mapping $T:X\rightrightarrows X^*$, we have 
\begin{equation}\label{T-T+J}
\gph T \overset{\Phi_{\sigma}}{\longmapsto} \gph (T+\sigma J)\overset{\Phi_{-\sigma}}{\longmapsto} \gph T .
\end{equation}
To illustrate \eqref{eq:shear}, look at the simplest case when $X=\R$. Here the vertical shear mapping takes the point $(x,y)$ to the point $(x^\prime,y^\prime)$, where 
\begin{equation*}
\begin{pmatrix}
x^\prime \\ y^\prime 
\end{pmatrix} = \begin{pmatrix}
1 & 0 \\ \sigma & 1 
\end{pmatrix}
\begin{pmatrix}
x \\ y
\end{pmatrix}.
\end{equation*}
The vertical shear displaces points to the right of the $y$-axis up or down, depending on the sign of $\sigma$. It leaves vertical lines invariant, but tilts all other lines around the point where they meet the $y$-axis. Horizontal lines, in particular, get tilted to become lines with slope $\sigma$.\vspace{0.03in}

To proceed further, observe that the (global and local) 
monotonicity and strong monotonicity notions for set-valued mappings can be applied to {\em sets} by just considering mappings whose graphs are exactly the sets. Having this in mind, we now present useful discussions involving set monotonicity via the vertical shear mapping \eqref{eq:shear}.\vspace*{-0.05in}

\begin{Remark}\label{phi}
\rm It follows from \eqref{T-T+J} that whenever $\sigma>0$, we have:

{\bf(i)} $\Phi_{\sigma}$ carries monotone sets into $\sigma$-strongly monotone sets. Moreover, we know from Proposition~\ref{prop:preser} that $\Phi_{\sigma}$ also maps maximal monotone sets into $\sigma$-strongly maximal monotone sets provided that $X$ is reflexive.

{\bf(ii)} If $W\subset X\times X^*$ is monotone with respect to $U\times V$, then $\Phi_{\sigma}(W)$ is $\sigma$-strongly monotone with respect to $\Phi_{\sigma}(U\times V)$.

{\bf(iii)} $\Phi_{-\sigma}$ carries $\sigma$-strongly monotone sets into monotone sets. Furthermore, Proposition~\ref{prop:preser} says that it also maps $\sigma$-strongly maximal monotone sets into maximal monotone sets.

{\bf(iv)} If $W\subset X\times X^*$ is $\sigma$-strongly monotone with respect to $U\times V$, then $\Phi_{-\sigma}(W)$ is monotone with respect to $\Phi_{-\sigma}(U\times V)$.

{\bf(v)} For any $U\times V \subset X\times X^*$ and for any $\sigma \neq 0$, it holds that  
\begin{eqnarray*}
\Phi_{\sigma}(U\times V)=\left\{(x,x^*) \in X\times X^* \;\big|\; x\in U,\ x^*-\sigma J (x) \in V\right\}.
\end{eqnarray*}
If furthermore $(X,\|\cdot\|)$ is Fr\'echet smooth and  $U\times V$ is open, then $ \Phi_{\sigma}(U\times V)$ is open as well.
\end{Remark}\vspace*{-0.03in}

Using the above definitions, properties of the vertical shear transformation $\Phi_{\sigma}$, and Proposition~\ref{prop:preser} allows us to verify the preservation of local maximal monotonicity of type $(B)$ under shifts.\vspace*{-0.05in}

\begin{Proposition}\label{prop:P-strP}
Let $(X,\|\cdot\|)$ be a Fr\'echet smooth space. Consider a set-valued mapping $T:X\rightrightarrows X^*$ and a graph point $(\ox,\ox^*)\in \gph T$. Then for any $\sigma >0$, we have:

{\bf(i)} If $X$ is reflexive and $T$ is locally maximal monotone around $(\ox,\ox^*)$, then $T+\sigma J$ is locally strongly maximal monotone of type $(B)$ around $(\ox,\ox^*+\sigma J( \ox))$ with modulus $\sigma$.

{\bf(ii)} If $T$ is locally strongly maximal monotone of type $(B)$ around $(\ox,\ox^*)$ with modulus $\sigma$, then $T-\sigma J$ is locally maximal monotone of type $(B)$ around $(\ox,\ox^*-\sigma J(\ox))$.
\end{Proposition}
\begin{proof}
To verify (i), we deduce from the assumptions in (i) that there exist a neighborhood $U\times V$ of $(\ox,\ox^*)$ and a maximal monotone operator $\bar{T}:X\rightrightarrows X^*$ such that
\begin{eqnarray*}
\gph \bar{T} \cap (U\times V) = \gph T \cap (U\times V).
\end{eqnarray*}
Then $\Phi_{\sigma} \big[ \gph \bar{T} \cap (U\times V)\big] = \Phi_{\sigma} \big[ \gph T \cap (U\times V)\big]$ while giving us
\begin{eqnarray}\label{TTbar}
\gph\big(\bar{T}+\sigma J\big) \cap \Phi_{\sigma}(U\times V) = \gph (T+\sigma J)\cap \Phi_{\sigma}(U\times V),
\end{eqnarray}
where $\Phi_{\sigma}(U\times V)$ is a neighborhood of $(\ox,\ox^*+\sigma J (\ox))$. It follows from Proposition~\ref{prop:preser}(ii) that $\bar{T}+\sigma J$ is $\sigma$-strongly maximal monotone on $X$. Thus we get from \eqref{TTbar} that $T+\sigma J$ is locally strongly maximal monotone of type $(B)$ around $(\ox,\ox^*+\sigma J(\ox))$ with modulus $\sigma$. Assertion (ii) is verified in the same way by using the shear mapping $\Phi_{-\sigma}$ and the result of Proposition~\ref{prop:preser}(i).
\end{proof}\vspace*{-0.05in}

Now we are in a position to reveal close relationships
between both type $(A)$ and $(B)$ of local strong maximal monotonicity of operators and single-valued continuous localizations of their inverses. Recall that the spaces in question are supposed to be G\^ateaux smooth. \vspace*{-0.03in}

\begin{Theorem}\label{coro:loc-Minty-strong}
Let $(X,\|\cdot\|)$ be a reflexive space, and let 
$T:X\rightrightarrows X^*$ be a set-valued mapping with $(\ox,\ox^*)\in \gph T$. For any $\sigma >0$, consider the following assertions:

{\bf(i)} $T$ is locally strongly maximal monotone of type $(B)$ around $(\ox,\ox^*)$ with modulus $\sigma$.

{\bf(ii)} $T$ is locally strongly monotone around $(\ox,\ox^*)$ with modulus $\sigma$ and its inverse $T^{-1}$ has a continuous single-valued localization around $(\ox^*,\ox)$.

{\bf(iii)} $T$ is locally strongly maximal monotone of type $(A)$ around $(\ox,\ox^*)$ with modulus $\sigma$.\\[0.5ex]
Then we always have the equivalence {\rm(i)$\Longleftrightarrow$(iii)}. The other equivalence {\rm(i)$\Longleftrightarrow$(ii)} holds provided that the norm $\|\cdot\|$ has the properties listed in \eqref{Troyanski-Asplund}.
\end{Theorem}
\begin{proof} Let us first verify that (i)$\Longrightarrow$(iii), where the given norm $\|\cdot\|$ is G\^ateaux differentiable on $X\setminus\{0\}$.
Supposing that $T$ is locally strongly maximal monotone of type $(B)$ around $(\ox,\ox^*)$ with modulus $\sigma>0$ tells us that $T$ is locally $\sigma$-strongly monotone and also locally maximal monotone of type $(B)$ around this point. We know from Theorem~\ref{theo:loc-Minty} that the latter is equivalent to the local maximal monotonicity of type $(A)$ around $(\ox,\ox^*)$. It follows from the definition that $T$ is in fact locally strongly maximal monotone of type $(A)$ around the reference point.

To prove the reverse implication (iii)$\Longrightarrow$(i), suppose that $T$ is locally strongly maximal monotone of type $(A)$ with modulus $\sigma>0$ relative to a neighborhood $U\times V$ of $(\ox,\ox^*)$. Obviously, the set $\gph T \cap (U\times V)$ is $\sigma$-strongly monotone. Using Proposition~\ref{lem:extend}(ii) under the imposed G\^ateaux smoothness and reflexivity of the space $(X,\|\cdot\|)$, we find a $\sigma$-strongly maximal monotone operator $\bar{T}:X\rightrightarrows X^*$ such that 
\begin{equation}\label{T--Tbar}
\gph T \cap (U\times V)\subset \gph \bar{T}.
\end{equation}
Since $\bar{T}$ is a monotone mapping and $T$ is locally strongly maximal monotone of type $(A)$ relative to $U\times V$, it follows from \eqref{T--Tbar} that $\gph T \cap (U\times V) = \gph \bar{T} \cap (U\times V)$. Therefore, we have a $\sigma$-strongly maximal monotone operator $\bar{T}$ on $X$ satisfying the condition $\gph T \cap (U\times V) = \gph \bar{T} \cap (U\times V)$, which means that $T$ is locally strongly maximal monotone of type $(B)$ with modulus $\sigma$ relative to $U\times V$. This completes the proof of the equivalence (i)$\Longleftrightarrow$(iii).\vspace*{0.03in}

Assuming now that the norm $\|\cdot\|$ satisfies all the geometric properties in \eqref{Troyanski-Asplund}, let us verify that (i)$\Longrightarrow$(ii). Indeed, it follows from Proposition~\ref{prop:P-strP}(ii) that the local strong maximal monotonicity of type $(B)$ in (i) ensures that the shifted operator $T-\sigma J$ is locally maximal monotone of type $(B)$ around $(\ox,\ox^*-\sigma J(\ox))$. Choosing $\lambda := \sigma^{-1}$ in implication (i)$\Longrightarrow$(ii) of Theorem~\ref{theo:loc-Minty} allows us to conclude that the mapping
\begin{equation*}
(\sigma^{-1}T)^{-1} = \big( J+ \sigma^{-1}(T-\sigma J)\big)^{-1}
\end{equation*}
admits a single-valued continuous localization around $(\ox^*/\sigma,\ox)$. Therefore, the inverse mapping $T^{-1}$ has a single-valued continuous localization around $(\ox,\ox^*)$, which justifies assertion (ii).

It remains to show that (ii) yields (i). Using (ii) and the discussions in Remark~\ref{phi} tells us that the shifted operator $T-\sigma J$ is locally monotone around $(\ox,\ox^*-\sigma J(\ox))$. We also have that $(\sigma^{-1}T)^{-1}$ admits a single-valued localization around $(\ox^*/\sigma,\ox)$. This ensures that the mapping
\begin{equation*}
\big(J+ \sigma^{-1}(T-\sigma J)\big)^{-1} =(\sigma^{-1}T)^{-1}
\end{equation*}
admits a single-valued localization around that point as well. Employing finally implication (ii)$\Longrightarrow$(i) in Theorem~\ref{theo:loc-Minty} verifies that $T-\sigma J$ is locally maximal monotone of type $(B)$ around $(\ox,\ox^*-\sigma J(\ox))$. Then assertion (i) follows from 
Proposition~\ref{prop:P-strP}(i), which completes the proof.
\end{proof}\vspace*{-0.03in}

Due to the equivalence of Theorem~\ref{coro:loc-Minty-strong} between the two types of local strong maximal monotonicity, we use the term ``{\em local strong maximal monotonicity}" when the space $X$ is reflexive.\vspace*{0.03in} 

The next proposition, which is a useful consequence of Theorem~\ref{coro:loc-Minty-strong} and some other results above, concerns relationships between local strong monotonicity and local maximal monotonicity of operators and their shifts in Fr\'echet smooth (reflexive and nonreflexive) Banach spaces.\vspace*{-0.05in}

\begin{Proposition}\label{coro:R-strongR}
Let $(X,\|\cdot\|)$ be a Fr\'echet smooth Banach space, and let $T:X\rightrightarrows X^*$ be a set-valued mapping with $(\ox,\ox^*)\in \gph T$. Then for any $\sigma>0$ the following hold:

{\bf(i)} If $X$ is reflexive and $T$ is locally maximal monotone around $(\ox,\ox^*)$, then the mapping $T+\sigma J$ is locally strongly maximal monotone around $(\ox,\ox^*+\sigma J(\ox))$ with modulus $\sigma$.

{\bf (ii)} If $T$ is locally strongly maximal monotone of type $(A)$ around $(\ox,\ox^*)$ with modulus $\sigma$, then $T-\sigma J$ is locally maximal monotone of the same type around $(\ox,\ox^*-\sigma J(\ox))$.
\end{Proposition}
\begin{proof} 
Assertion (i) is a direct consequence of Proposition~\ref{prop:P-strP}(i), and Theorem~\ref{coro:loc-Minty-strong}. To verify  (ii), suppose that $T$ is strongly maximal monotone of type $(A)$ on a neighborhood $U\times V$ of $(\ox,\ox^*)$ with modulus $\sigma$ and then prove that $T-\sigma J$ is  maximal monotone of the same type with respect to the neighborhood 
\begin{eqnarray*}
\Phi_{-\sigma}(U\times V)= \left\{(x,x^*)\in X\times X^* \;\big|\; x\in U,\ x^*+\sigma J (x)\in V\right\}
\end{eqnarray*}
of $(\ox,\ox^*-\sigma J(\ox))$. First we deduce from  the $\sigma$-strong monotonicity of $T$ with respect to $U\times V$ and the discussions in Remark~\ref{phi} that $T-\sigma J$ is monotone with respect to $\Phi_{-\sigma}(U\times V)$. To justify its local maximality of type $(A)$, consider an operator $S:X\rightrightarrows X^*$, which is monotone with respect to $\Phi_{-\sigma}(U\times V)$ and such that $\gph (T-\sigma J)\cap  \Phi_{-\sigma}(U\times V) \subset \gph S \cap \Phi_{-\sigma}(U\times V)$. Taking the image of $\Phi_{\sigma}$ on both sides of this inclusion brings us to 
\begin{eqnarray}\label{eq:qq}
\gph T\cap (U\times V) \subset \gph (S+\sigma J) \cap (U\times V). 
\end{eqnarray}
Since $S+\sigma J$ is globally monotone and $T$ is maximal monotone with respect to $U\times V$, we get from \eqref{eq:qq} that $\gph T\cap (U\times V) = \gph (S+\sigma J)\cap (U\times V)$. Taking now the image of $\Phi_{-\sigma}$ on both sides of the latter equality yields 
$$
\gph (T-\sigma J) \cap \Phi_{-\sigma}(U\times V) = \gph S \cap \Phi_{-\sigma}(U\times V),
$$ 
which therefore completes the proof of the proposition.
\end{proof}\vspace*{-0.03in}

Finally in this section, we present a consequence of Proposition~\ref{coro:R-strongR} on the $\sigma$-strong local maximal monotonicity of operators without involving the shift mappings in the formulation while using them in the proof. This corollary is a local counterpart to Proposition~\ref{strong-alter-glo}.\vspace*{-0.05in}

\begin{Corollary}\label{coro:alter-R}
Let $(X,\|\cdot\|)$ be a reflexive and Fr\'echet smooth Banach space, and let $T:X\tto X^*$ with $(\ox,\ox^*)\in \gph T$. Then for any $\sigma>0$ we have that $T$ is locally strongly maximal monotone around $(\ox,\ox^*)$ with modulus $\sigma$ if and only if there is a neighborhood $U\times V$ of $(\ox,\ox^*)$ such that $T$ is $\sigma$-strongly monotone with respect to $U\times V$, and that 
\begin{equation}\label{R-equ}
\gph T\cap (U\times V) = \gph S\cap (U\times V)
\end{equation}
for any $\sigma$-strongly monotone mapping $S:X\rightrightarrows X^*$ on $U\times V$ satisfying the inclusion $\gph T\cap (U\times V)\subset \gph S \cap (U\times V)$.
\end{Corollary}
\begin{proof}
The ``only if" part is trivial. To verify  the ``if" part, assume the existence of a neighborhood $U\times V$ of $(\ox,\ox^*)$ such that $T$ is $\sigma$-strongly monotone with respect to $U\times V$ while satisfying the equality in \eqref{R-equ} for any $\sigma$-strongly monotone operator $S:X\tto X^*$ on $U\times V$ with $\gph T\cap (U\times V)\subset \gph S \cap (U\times V)$. We intend to check that $T$ is $\sigma$-strongly maximal monotone of type $(A)$ with respect to $U\times V$, which gives us the claimed result due to the equivalence in Theorem~\ref{coro:loc-Minty-strong}. Indeed, utilizing the vertical shear transformation $\Phi_{-\sigma}$ tells us that the shifted operator $T-\sigma J$ is maximal monotone of type $(A)$ with respect to the set $\Phi_{-\sigma}(U\times V)$, which is a neighborhood of $(\ox,\ox^*-\sigma J(\ox))$. Then it follows from Proposition~\ref{coro:R-strongR}(i) that $T$ is locally strongly maximal monotone of type $(A)$ around $(\ox,\ox^*)$, and we are done with the proof.
\end{proof}\vspace*{-0.2in}

\section{Generalized Differentiation}\label{sec:gen-diff}\vspace*{-0.05in}

Our further plans in this paper include the usage of appropriate {\em generalized differential tools} of variational analysis to {\em characterize} local maximal monotonicity of set-valued mappings. This will be done in Section~\ref{sec:coderi-localmax} in the framework of Hilbert spaces. In what follows, all the spaces under consideration are assumed to be {\em Hilbert} without special mentioning. 

This section overviews the constructions and preliminaries of generalized differentiation in Hilbert spaces that are needed for local maximal monotonicity characterizations in Section~\ref{sec:coderi-localmax}. We refer the reader to \cite{Mordukhovich06} for the corresponding extensions and further material in more generality, and to \cite{Mordukhovich18,Rockafellar98} for rich variational theory and applications in finite dimensions.\vspace*{0.03in}

Our main generalized differential characterizations of local maximal monotonicity and its strong counterparts are obtained in terms of {\em coderivatives} of set-valued mappings. To begin with, we define the (Fr\'echet) {\em regular normal cone} to a set $\O\subset X$ at $\ox\in\O$  by
\begin{equation}\label{rnc}
\Hat{N}(x;\Omega):=\Big\{x^*\in X\;\Big|\;\limsup_{u\overset{\Omega}{\rightarrow}x}\frac{\langle
x^*,u-x\rangle}{\|u-x\|}\le 0\Big\}. 
\end{equation} 
Among various generalized differential constructions for mappings, in this paper we use the following two coderivatives. Given a set-valued mapping $F\colon X\tto Y$ between Hilbert spaces, the {\em regular coderivative} of $F$ at $(\ox,\oy)\in\gph F$ is defined by 
\begin{equation}\label{reg-cod} 
\Hat
D^*F(\ox,\oy)(y^*):=\big\{x^*\in X\;\big|\;(x^*,-y^*)\in\Hat N((\ox,\oy);\gph F)\big\},\quad y^*\in Y,
\end{equation}
via the regular normal cone \eqref{rnc} to the graph of $F$. 
The other coderivative construction for $F$ at $(\ox,\oy)\in\gph F$, known as the (Mordukhovich) {\em mixed coderivative}, is given by: 
\begin{equation}\label{mix-lim-cod}
\begin{array}{ll}
D^*_M F(\ox,\oy)(\bar{y}^*):=\Big\{x^*\in X\;\Big|&\exists\,
(x_k,y_k,y_k^*)\rightarrow(\bar{x},\bar{y},\bar{y}^*),\;x_k^*\xrightarrow{w}\bar{x}^*\;\mbox{ as }\;k\to\infty\\
&\mbox{with }\;(x_k,y_k)\in \gph F,\;x_k^*\in \widehat{D}^*F(x_k,y_k)(y_k^*)\Big\},\quad\oy^*\in Y.
\end{array}
\end{equation}
Observe that the limiting construction in \eqref{mix-lim-cod} employs the {\em mixed} convergence of $(x^*_k,y^*_k)$: weak for $x^*_k$ and strong for $y^*_k$. This differs \eqref{mix-lim-cod} from the {\em normal coderivative} $D^*_N F(\ox,\oy)(\bar{y}^*)$ in infinite dimensions (which is not used in this paper), where both $x^*_k$ and $y^*_k$ converge weakly on $X\times Y$. Despite the nonconvexity of $D^*_M F(\ox,\oy)(\bar{y}^*)$ and $D^*_N F(\ox,\oy)(\bar{y}^*)$, both coderivatives enjoy {\em full calculus} based on the fundamental {\em extremal principle} of variational analysis. They both play a prominent role in variational analysis and numerous applications; see \cite{Mordukhovich06}. Note that when $F$ is single-valued and continuously differentiable around $\ox$, then we have
\begin{equation*} \Hat{D}^*F(\bar{x})(y^*)=D^*_M F(\bar{x})(y^*)=D^*_N F(\bar{x})(y^*)=\big\{\nabla F(\bar{x})^*y^*\big\},\quad y^*\in Y,
\end{equation*} 
where $\nabla F(\bar{x})^*$ stands for the adjoint derivative operator, and where $\oy=F(\ox)$ is omitted.\vspace*{0.03in}

To conclude this section, we formulate the following simple observation used in Section~\ref{sec:coderi-localmax}.\vspace*{-0.05in}

\begin{Proposition}\label{prop:GF}
Let $F,G:X\rightrightarrows Y$ be set-valued mappings with a common point $(\ox,\oy)$ of their graphs. Assume that there exists a neighborhood $W$ of $(\ox,\oy)$ such that
\begin{equation*}
\gph F\cap W=\gph G\cap W.
\end{equation*}
Then for all $y^*\in Y$ we have the equalities
\begin{equation*}
\widehat{D}^*F(\ox,\oy)(y^*) = \widehat{D}^* G(\ox,\oy)(y^*)\;\mbox{ and }\;D^*_M F(\ox,\oy)(y^*) = D^*_M G(\ox,\oy)(y^*).
\end{equation*}
\end{Proposition}
\begin{proof} It suffices to check the first equality while the second one follows from definition \eqref{mix-lim-cod}. To proceed, pick $x^*\in \widehat{D}^*F(\ox,\oy)(y^*)$, i.e., $(y^*,x^*) \in\gph\widehat{D}^*F(\ox,\oy)$. We get from \eqref{reg-cod} that
\begin{eqnarray*}
(x^*,-y^*) \in\widehat{N}\big((\ox,\oy);\gph F\big)=\widehat{N}\big((\ox,\oy);\gph F\cap W\big)&=&\widehat{N}\big((\ox,\oy);\gph G\cap W\big)\\
&=&\widehat{N}\big((\ox,\oy);\gph G\big),
\end{eqnarray*}
where the first and last equalities are due to definition \eqref{rnc} and $(\ox,\oy)\in W$, where the neighborhood $W$ is an open set. This gives us $(x^*,-y^*)\in\widehat{N}((\ox,\oy);\gph G)$, which means that $(y^*,x^*)\in \gph\widehat{D}^*G(\ox,\oy)$ as claimed.
\end{proof}\vspace*{-0.15in}

\section{Coderivative Criteria for Local Maximal Monotonicity}\label{sec:coderi-localmax}\vspace*{-0.05in}

The main results of this section provide {\em necessary and sufficient conditions} for {\em local maximal monotonicity} of set-valued mappings on Hilbert spaces obtained in terms of their regular and mixed {\em coderivatives}. As a consequence, we derive also coderivative characterizations of {\em local strong} maximal monotonicity. The resolvent (Minty-type) characterizations of local maximal monotonicity obtained in Theorem~\ref{theo:loc-Minty}, being married to {\em coderivative calculus} and criteria for {\em Lipschitzian stability}, plays a crucial role in our device.\vspace*{0.03in}

In our characterizations below, we use certain hypomonotonicity properties of set-valued mappings, which are recalled now for mappings in Hilbert spaces. It is said that $T\colon X\tto X$ is (globally) {\em hypomonotone} with modulus $r>0$ if the mapping $T+r I$ is monotone on $X$, i.e.,
\begin{equation*}
\la x^*_1-x^*_2,x_1-x_2\ra\ge-r\|x_1-x_2\|^2\;\mbox{ for all }\;(x_1,x^*_1),(x_2,x^*_2)\in X\times X.  
\end{equation*}
The {\em local hypomonotonicity} of $T$ around a point $(\ox,\ox^*)\in\gph T$ with modulus $r>0$ is defined via the existence of a neighborhood $U\times V$ of $(\ox,\ox^*)$ and a globally hypomonotone operator $\bar T\colon X\tto X$ with modulus $r>0$ such that
\begin{equation*}
\gph T\cap(U\times V)=\gph\bar T\cap(U\times V).
\end{equation*}
It has been realized in variational analysis that the hypomonotonicity properties hold for broad classes of mappings that appear in various problems of optimization, control, and numerous applications. The reader is referred to \cite{MordukhovichNghia1,Pen02,Rockafellar98} and the bibliographies therein for more details.\vspace*{0.03in}

We begin with formulating the following characterization of {\em global} maximal monotonicity, which is taken from \cite[Theorem~3.7]{CBN} and is used in the proof of the main result below.\vspace*{-0.05in}

\begin{Lemma}\label{prop:mixed} Let $T:X\rightrightarrows X$ be a set-valued mapping with closed graph. Then we have the equivalent assertions:

{\bf(i)} $T$ is maximal monotone on $X$.

{\bf(ii)} $T$ is hypomonotone on $X$ with some modulus $r>0$, and the mixed coderivative $D^*_M T(u,v)$ is positive-semidefinite on $X$ meaning that
\begin{eqnarray*}
\la z,w \ra \ge 0\;\text{ for any }\;z\in D^*_M T(u,v)(w),\;(u,v)\in\gph T,\;\mbox{ and }\;w\in X.
\end{eqnarray*}
\end{Lemma}

To derive coderivative characterizations of {\em local} maximal monotonicity, we employ a novel {\em vertical transvection} technique of its own interest. Taking any real number $\sigma\neq 0$, define the transformation $\Delta_{\sigma}:X\times X\rightarrow X\times X$ by 
\begin{equation}\label{shear+reflect}
\Delta_{\sigma}(x,v) := (v+\sigma x,x)\;\mbox{ for any }\;(x,v)\in X\times X.
\end{equation}
Observe that $\Delta_{\sigma} = \Theta \circ \Phi_\sigma$ is the composition of the vertical shear mapping $\Phi_{\sigma}$ from \eqref{eq:shear} and the reflection $\Theta$ over the diagonal $\{(u,u)\in X\times X\mid u\in X\}$ in the scheme
\begin{equation*}
(x,v)\overset{\Phi_{\sigma}}{\mapsto} (x,v+\sigma x) \overset{\Theta}{\mapsto} (v+\sigma x,x).
\end{equation*}
It is easy to check the following properties of $\Delta_{\sigma}$:\\[1ex]
{\bf(a)} $\Delta_{\sigma}$ is a homeomorphism from $X\times X$ to itself as the composition of two homeomorphisms.\\[1ex]
{\bf(b)} The inverse map $\Delta_{\sigma}^{-1} = \Phi_{\sigma}^{-1}\circ \Theta^{-1} = \Phi_{-\sigma}\circ \Theta$ is specified as
\begin{equation}\label{Delta-inv}
(v,x) \overset{\Theta}{\mapsto} (x,v) \overset{\Phi_{-\sigma}}{\mapsto} (x,v-\sigma x).
\end{equation}
{\bf(c)} $\Delta_\sigma$ maps the graph of any multifunction $T:X\rightrightarrows X$ into the graph of $(T+\sigma I)^{-1}$, i.e.,
\begin{equation}\label{Delta}
\Delta_\sigma (\gph T)=\gph (T+\sigma I)^{-1},
\end{equation}
while $\Delta_{\sigma}^{-1}\big(\gph (T+\sigma I)^{-1}\big) = \gph T$.\vspace*{0.05in}

Utilizing the transformations $\Delta_{\sigma}$ and its inverse, we get the following proposition.\vspace*{-0.05in}

\begin{Proposition}\label{transform} Let $T:X\rightrightarrows X$ be a set-valued mapping with $(\ox,\ov)\in \gph T$ and such that $\gph T$ is locally closed around $(\ox,\ov)$. Suppose that $T$ is locally hypomonotone around $(\ox,\ov)$ with modulus $r>0$. Then for all $\sigma >r$, the mapping $(T+\sigma I)^{-1}$ has a localization $R_{\sigma}$ around $(\ov+\sigma \ox,\ox)$, which is single-valued on its domain satisfying the implication
\begin{equation} \label{coderivative}
\big[w\in\widehat{D}^*R_\sigma (v)(z) \Longrightarrow -z + \sigma w \in \widehat{D}^*T(R_\sigma(v),v -\sigma R_\sigma (v))(-w)\big] \ \text{ for all }\ v\in \dom R_{\sigma}.
\end{equation} 
\end{Proposition}
\begin{proof} 
As $T$ is locally hypomonotone around $(\ox,\ov)$ with modulus $r>0$, there exists a neighborhood $W\subset X\times X$ of $(\ox,\ov)$ such that 
\begin{equation} \label{hypolocalmono}
\langle v_1 -v_2,x_1-x_2\rangle \geq -r\|x_1-x_2\|^2 \quad \text{for all }\; (x_1,v_1),(x_2,v_2)\in\text{\rm gph}T\cap W.
\end{equation} 
Pick any $\sigma >r$. Since $\Delta_{\sigma}$ is continuous with $(\ov+\sigma \ox,\ox)=\Delta_\sigma (\ox,\ov)$, we have from \eqref{Delta} that $(\ov+\sigma \ox,\ox)\in \gph (T+\sigma I)^{-1}$, and $\Delta_\sigma (W)$ is a neighborhood of $(\ov+\sigma \ox,\ox)$ as $W$ is open. Let $R_{\sigma}$ be the graphical localization of $(T+\sigma I)^{-1}$ with respect to $\Delta_\sigma (W)$ given by
\begin{equation}\label{R_lam}
\gph R_\sigma = \gph (T+\sigma I)^{-1} \cap \Delta_\sigma (W).
\end{equation}
We claim that $R_\sigma$ is single-valued on its domain. For any $(v_i,x_i)\in \text{\rm gph}R_\sigma$, $i=1,2$, it follows from \eqref{Delta-inv}, \eqref{Delta}, and \eqref{R_lam} that
\begin{equation*}
(x_i,v_i-\sigma x_i) = \Delta^{-1}_{\sigma} (v_i,x_i) \in \Delta^{-1}_{\sigma} (\gph R_{\sigma}) = \Delta^{-1}_{\sigma}\big(\gph (T+\sigma I)^{-1}\big) \cap \Delta^{-1}_{\sigma}\big( \Delta_\sigma (W)\big) = \gph T \cap W.
\end{equation*}
Combining the latter with \eqref{hypolocalmono} tells us that 
$$
\langle v_1 -\sigma x_1 - (v_2 -\sigma x_2), x_1- x_2\rangle \ge -r \|x_1-x_2\|^2,
$$
which implies in turn that 
$$
\langle v_1 -v_2,x_1 -x_2 \rangle \geq (\sigma - r)\|x_1-x_2\|^2.
$$
Applying the Cauchy-Schwarz inequality yields the estimate
$$
\|v_1-v_2\|\geq (\sigma - r)\|x_1-x_2\|,
$$
which implies $R_\sigma (v)$ has merely one element for all $v \in \dom R_{\sigma}$.\vspace*{0.03in}

Next we justify \eqref{coderivative}. Picking $v \in \dom R_\sigma$ and $w \in \widehat{D}^*R_\sigma (v) (z)$ gives us by \eqref{reg-cod} that
$$
(w,-z)\in \widehat{N}\big((v,R_{\sigma}(v));\gph R_{\sigma}\big).
$$
It is easy to deduce from definition \eqref{rnc} of the regular normal cone that 
$$
\begin{aligned}
(w,-z)\in \widehat{N}\big((v,R_{\sigma}(v));\gph R_{\sigma}\big) &= \widehat{N}\big((v,R_{\sigma}(v));\gph (T+\sigma I)^{-1} \cap \Delta_{\sigma}(W)\big) \\
&= \widehat{N}\big((v,R_{\sigma}(v));\gph (T+\sigma I)^{-1}\big)
\end{aligned}
$$
since $\Delta_{\sigma}(W)$ is an open set containing $(v,R_{\sigma}(v))$. Therefore,
$$
w \in \widehat{D}^*(T+\sigma I)^{-1}(v, R_\sigma(v))(z),
$$
which can be clearly rewritten as 
\begin{equation}\label{P1}
-z \in \widehat{D}^* (T+\sigma I)(R_\sigma (v), v)(-w).
\end{equation}
Employing the coderivative sum rule from \cite[Theorem~1.62(i)]{Mordukhovich06}) says that
\begin{eqnarray}
\widehat{D}^* (T+\sigma I)(R_\sigma (v), v)(-w) &=& \widehat{D}^* T(R_\sigma (v), v- \sigma R_\sigma (v))(-w) + \nabla (\sigma I)(R_{\sigma}(v))^* (-w) \notag \\
&=& \widehat{D}^* T(R_\sigma (v), v- \sigma R_\sigma (v))(-w)-\sigma w. \label{P2}
\end{eqnarray}
Finally, it follows from \eqref{P1} and \eqref{P2} that $-z +\sigma w \in \widehat{D}^* T(R_\sigma (v), v- \sigma R_\sigma (v))(-w)$, which completes the proof of the proposition. 
\end{proof}\vspace*{-0.03in}

Before establishing the main result of this section, we need to recall one more property of multifunction. Namely, $F:X\rightrightarrows X$ is {\em Lipschitz-like} (or has the Aubin property) around $(\ox,\ov)\in \gph F$ with modulus $\ell \ge 0$ if there are neighborhoods $U$ of $\ox$ and $V$ of $\ov$ such that 
\begin{equation}\label{defi:Aubin}
F(x) \cap V \subset F(u) + \ell \|x-u\| \mathbb{B}\;\text{ for all }\;x,u\in U,
\end{equation}
where $\B$ stands for the unit ball in $X$. Variational analysis achieves complete characterizations of the Lipschitz-like property in both finite and infinite dimensions known as the coderivative/Mordukhovich criteria; see \cite{Mordukhovich06,Rockafellar98} and the reference therein.\vspace*{0.03in}

Now we are ready to derive the aforementioned coderivative characterizations of local maximal monotonicity of set-valued mappings in Hilbert spaces.\vspace*{-0.05in}

\begin{Theorem}\label{theo:main} Let $T:X\rightrightarrows X$ be of locally closed graph around $(\ox,\ov)\in \gph T$. Then the following assertions are equivalent:

{\bf(i)} $T$ is locally maximal monotone around $(\ox,\ov)$.

{\bf(ii)} $T$ is locally hypomonotone around $(\ox,\ov)$ with some modulus $r>0$, and there exists a neighborhood $W$ of $(\ox,\ov)$ such that  
\begin{eqnarray}\label{eq:5.2D}
\la z,w\ra \ge 0\ \text{ whenever }\ z\in D^*_M T(u,v)(w)\ \text{ and } \ (u,v)\in \gph T \cap W.
\end{eqnarray}

{\bf(iii)} $T$ is locally hypomonotone around $(\ox,\ov)$ with some modulus $r>0$, and there exists a neighborhood $W$ of $(\ox,\ov)$ such that
\begin{eqnarray}\label{eq:5.2*}
\la z,w\ra \ge 0\ \text{ whenever }\ z\in \Hat{D}^* T(u,v)(w)\ \text{ and } \ (u,v)\in \gph T \cap W.
\end{eqnarray}
\end{Theorem}\vspace*{-0.2in}
\begin{proof} 
To verify (i)$\Longrightarrow$(ii), we can suppose by the equivalence of Theorem~\ref{theo:loc-Minty} that $T$ is locally maximal monotone of type $(B)$ relative to a neighborhood $W$ of $(\ox,\ov)$. Thus Definition~\ref{defi:local-max-R}(ii) gives us a globally maximal monotone mapping $G:X\rightrightarrows X$ of closed graph such that
\begin{eqnarray}\label{GT}
\gph G \cap W = \gph T \cap W.
\end{eqnarray}
Pick $(u,v)\in \gph T \cap W$ and $(z,w)\in X\times X$ with $z\in D^*_M T(u,v)(w)$ and get by \eqref{GT} and Proposition~\ref{prop:GF} that $z\in D^*_M G(u,v)(w)$. Combining the latter with $(u,v)\in \gph G$ and the maximal monotonicity of $G$ on $X$, we deduce from Lemma~\ref{prop:mixed} that $\langle z,w\rangle \ge 0$, which justifies the claimed implication. The next implication (ii)$\Longrightarrow$(iii) immediately follows from the fact that the regular coderivative is always contained in the mixed one. 

Let us now prove the most challenging implication (iii) $\Longrightarrow$(i). Assuming that $T$ is locally hypomonotone around $(\ox,\ov)$ with modulus $r>0$ and the fulfillment of condition \eqref{eq:5.2*}, we deduce from Proposition~\ref{transform} that for all $\sigma >r$ the set-valued mapping $R_\sigma: X\rightrightarrows X$ defined in \eqref{R_lam} is single-valued on $\dom R_\sigma$, and that for all $y\in \dom R_{\sigma}$ we have the implication 
\begin{equation}\label{inclucoderivative}
w \in \widehat{D}^*R_\sigma (y)(z) \Longrightarrow -z + \sigma w \in \widehat{D}^*T(R_\sigma(y),y -\sigma R_\sigma (y))(-w),  
\end{equation} 
where $\Delta_{\sigma}(u,v)$ is taken from \eqref{shear+reflect}. We split the rest of the proof into three claims.\\[1ex]
{\bf Claim~1:} {\em $R_\sigma$ is Lipschitz-like \eqref{defi:Aubin} around $(\ov + \sigma \ox,\ox)$ with modulus $\sigma^{-1}$}. To justify this claim, pick $(y,R_{\sigma}(y))\in \gph R_{\sigma}\cap \Delta_{\sigma}(W)$ and $w\in \widehat{D}^*R_{\sigma}(y)(z)$ remembering that $\Delta_\sigma(W)$ is a neighborhood of $(\ov + \sigma \ox,\ox)$. It follows from \eqref{inclucoderivative} that
\begin{equation}\label{5.6i}
-z + \sigma w \in \widehat{D}^*T(R_\sigma(y),y -\sigma R_\sigma (y))(-w)  .
\end{equation}
By $(y,R_{\sigma}(y))\in\Delta_{\sigma}(W)$, we also observe that 
\begin{equation}\label{5.6ii}
(R_\sigma(y),y -\sigma R_\sigma (y))=\Delta^{-1}_{\sigma}(y,R_{\sigma}(y))\in\Delta^{-1}_{\sigma}(\Delta_{\sigma}(W))= W.
\end{equation}
Combining \eqref{eq:5.2*}, \eqref{5.6i}, and \eqref{5.6ii} gives us $\langle z, w\rangle \geq \sigma \|w\|^2$, and hence
\begin{equation}\label{Lips}
\|w\| \le\sigma^{-1}\|z\|\;\text{ whenever }\; w \in   \widehat{D}^*R_\sigma (y)(z).
\end{equation}
Since \eqref{Lips} holds for all $(y,R_{\sigma}(y))\in \gph R_{\sigma} \cap \Delta_{\sigma}(W)$ and all $z\in X$ with $w\in \widehat{D}^*R_{\sigma}(y)(z)$, it follows from the regular coderivative criterion in \cite[Theorem~4.7]{Mordukhovich06} that $R_\sigma $ is Lipschitz-like around $(\ov + \sigma \ox,\ox)$ with modulus $\sigma^{-1}$ as claimed.\\[1ex]
{\bf Claim~2:} {\em $T$ is locally monotone around $(\ox,\ov)$.} Indeed, it follows from Claim~1 that there exist neighborhoods $V$ of $\ov+\sigma \ox$ and $U$ of $\ox$ for which 
\begin{equation}\label{Llike}
R_\sigma (y_1) \cap U \subset R_\sigma (y_2) + \sigma^{-1}\|y_1-y_2\|\mathbb{B}\;\text{ whenever }\;y_1,y_2 \in V.  
\end{equation}
Furthermore, we have the condition 
\begin{equation}\label{localizationmono}
R_\sigma (y) \cap U \ne\emptyset\;\text{ for all }\;y \in V.
\end{equation}
As $\Delta_{\sigma}^{-1}(\ov+\sigma \ox,\ox)=(\ox,\ov)$ and $W$ is an open set containing $(\ox,\ov)$, we may shrink $U$ and $V$ if necessary such that $\Delta_{\sigma}^{-1}(V\times U)\subset W$. Picking two pairs $(u_i,v_i)\in\text{\rm gph}T\cap \Delta_{\sigma}^{-1}(V\times U)$, $i=1,2$, brings us to the equalities 
\begin{eqnarray*}
(v_i+\sigma u_i,u_i) &=& \Delta_{\sigma}(u_i,v_i)\in \Delta_{\sigma} \big( \gph T \cap \Delta_{\sigma}^{-1}(V\times U)\big) =\Delta_{\sigma} (\gph T) \cap \Delta_{\sigma}\big(\Delta_{\sigma}^{-1}(V\times U)\big) \\
&=& \gph (T+\sigma I)^{-1} \cap (V\times U) = \gph (T+\sigma I)^{-1} \cap (V\times U) \cap \Delta_{\sigma}(W) \\
&=& \gph R_\sigma \cap (V\times U).
\end{eqnarray*}
Since $R_\sigma (y)$ has at most one element for each $y \in X$, it follows from \eqref{Llike} that
$$
u_1 \in u_2 + \sigma^{-1}\|(v_1+\sigma u_1)-(v_2+\sigma u_2)\|\mathbb{B}
$$
which implies in turn that 
$$
\|u_1- u_2\| \le\sigma^{-1}\|v_1-v_2+\sigma(u_1-u_2)\|.
$$
Therefore, we arrive at the estimate
$$
\sigma^2\|u_1-u_2\|^2 \leq \|v_1-v_2\|^2 +2\sigma \langle v_1 - v_2 , u_1 -u_2\rangle + \sigma^2 \|u_1-u_2\|^2,
$$
which consequently gives us the condition 
$$
\langle v_1-v_2,u_1-u_2\rangle \ge-(2\sigma)^{-1}\|v_1-v_2\|^2 .
$$
By the arbitrary choice of $\sigma\in(r,\infty)$, we can let $\sigma \to \infty$ in the above inequality and get $$
\langle v_1-v_2,u_1-u_2\rangle \geq 0,
$$ 
which justifies the local  monotonicity of $T$ around $(\ox,\ov)$.\\[1ex]
{\bf Claim~3:}  {\em The mapping $(T+\sigma I)^{-1}$ admits a single-valued localization around $(\ov+\sigma \ox,\ox)$.} 
Since $V\times U\subset \Delta_{\sigma}(W)$ and the mapping $R_{\sigma}$ defined in \eqref{R_lam} is single-valued on its domain, we conclude that $R_{\sigma}|_V$ is also single-valued due to \eqref{localizationmono}, which justifies Claim~3. Employing finally Theorem~\ref{theo:loc-Minty} ensures that $T$ is locally maximal monotone around $(\ox,\ov)$. 
\end{proof}\vspace*{-0.05in}

As a consequence of Theorem~\ref{theo:main}, we now derive the following complete coderivative characterizations of {\em local strong maximal monotonicity} of set-valued operators in Hilbert spaces. Note that the results below have been obtained in \cite{MordukhovichNghia1} independently of the characterizations of local maximal monotonicity in Theorem~\ref{theo:main} by using an essentially more involved device.\vspace*{-0.05in} 

\begin{Corollary}\label{coro:strong-code}
Let $T:X\rightrightarrows X$ be a set-valued mapping of closed graph around $(\ox,\ov)\in \gph T$. Then for any $\sigma >0$, we have the equivalent assertions:

{\bf(i)} $T$ is locally strongly maximal monotone around $(\ox,\ov)$ with modulus $\sigma$.

{\bf(ii)} $T$ is locally hypomonotone around $(\ox,\ov)$, and there is a neighborhood $W$ of $(\ox,\ov)$ with \begin{eqnarray*}
\la z,w\ra \ge \sigma \|w\|^2\ \text{ whenever }\ z\in D^*_M T(u,v)(w)\ \text{ and } \ (u,v)\in \gph T \cap W.
\end{eqnarray*}

{\bf(iii)} $T$ is locally hypomonotone around $(\ox,\ov)$, and there is a neighborhood $W$ of $(\ox,\ov)$ with \begin{eqnarray}\label{eq:D*T}
\la z,w\ra \ge \sigma \|w\|^2\ \text{ whenever }\ z\in \Hat{D}^*T(u,v)(w)\ \text{ and } \ (u,v)\in \gph T \cap W.
\end{eqnarray}
\end{Corollary}
\begin{proof} Assume (i) and show that (ii) holds. Indeed, Proposition~\ref{prop:P-strP}(ii) implies that $T-\sigma I$ is locally maximal monotone around $(\ox,\ov-\sigma \ox)$. Applying Theorem~\ref{theo:main} gives us a neighborhood $Q$ of $(\ox,\ov-\sigma \ox)$ such that 
\begin{equation}\label{7.5}
\la x,y\ra \ge 0\ \text{ for all }\ x\in D^*_M (T-\sigma I)(s,t)(y)\ \text{ and } \ (s,t)\in \gph (T-\sigma I) \cap Q. 
\end{equation}
Since $\Phi_{\sigma}(\ox,\ov-\sigma \ox) = (\ox,\ov)$ and the set $Q\ni (\ox,\ov-\sigma \ox)$ is open, its image $W:=\Phi_{\sigma}(Q)$ is a neighborhood of $(\ox,\ov)$. Picking $(u,v)\in \gph T \cap W$ and $z\in D^*_M T(u,v)(w)$, let us check that $\la z,w\ra \ge \sigma \|w\|^2$ by using \eqref{7.5}. To proceed, for $(u,v)\in \gph T \cap W$ we have
\begin{eqnarray*}
(u,v-\sigma u) &=& \Phi_{-\sigma}(u,v)\in \Phi_{-\sigma}(\gph T \cap W) \\
&=& \Phi_{-\sigma}(\gph T) \cap \Phi_{-\sigma} (W) = \gph (T-\sigma I) \cap \Phi_{-\sigma}(\Phi_{\sigma}(Q)) = \gph (T-\sigma I) \cap Q.
\end{eqnarray*}
Employing the coderivative sum rule from \cite[Theorem~3.10]{Mordukhovich18} with $F_1:=T-\sigma I$, $F_2:=I$, $(u,v)\in \gph T$, and $w\in X$ ensures that
\begin{eqnarray*}
\begin{array}{ll}
D^*_M T(u,v)(w)&\subset D^*_M (T-\sigma I)(u,v-\sigma u)(w) + D^*_M (\sigma I) (u,\sigma u)(w)\\
&= D^*_M (T-\sigma I)(u,v-\sigma u)(w) +\sigma w,
\end{array}
\end{eqnarray*}
which results in the inclusions 
$$
z-\sigma w\in D^*_M T(u,v)(w)-\sigma w \subset D^*_M (T-\sigma I)(u,v-\sigma u)(w).
$$ 
As $(u,v-\sigma u)\in \gph (T-\sigma I) \cap Q$ and $z-\sigma w \in D^*_M (T-\sigma I)(u,v-\sigma u)(w)$, \eqref{7.5} yields
\begin{eqnarray*}
\la z-\sigma w,w\ra \ge 0,\;\mbox{ i.e.,}\;\la z,w\ra \ge \sigma\|w\|^2,
\end{eqnarray*}
and thus justifies (i)$\Longrightarrow$(ii). Since (ii) $\Longrightarrow$(iii) is trivial, it remains to verify 
(iii)$\Longrightarrow$(i).

To this end, suppose that $T$ is locally hypomonotone around $(\ox,\ov)$ and there exists a neighborhood $W$ such that \eqref{eq:D*T} holds. The hypomonotonicity of $T$ around $(\ox,\ov)$ ensures the one for $T-\sigma I$ around $(\ox,\ov-\sigma \ox)$. By the regular coderivative sum rule from \cite[Theorem~1.62(i)]{Mordukhovich06} and elementary calculations, we check that 
\begin{equation}\label{D^T-}
\la x,y\ra \ge 0 \ \text{ whenever } \ x\in \widehat{D}^* (T-\sigma I)(s,t)(y) \ \text{ and } \ (s,t)\in \gph (T-\sigma I) \cap \Phi_{-\sigma}(W).
\end{equation}
Due to \eqref{D^T-} and the local hypomonotonicity of $T-\sigma I$ around $(\ox,\ov-\sigma \ox)$, Theorem~\ref{theo:main} gives us the local maximal monotonicity of $T-\sigma I$ around $(\ox,\ov-\sigma \ox)$. Finally, we deduce from Proposition~\ref{coro:R-strongR}(i) that the original mapping $T$ is locally strongly maximal monotone around $(\ox,\ov)$ with modulus $\sigma$, which therefore completes the proof. 
\end{proof}\vspace*{-0.2in}

\section{Concluding Remarks}\label{sec:conc}\vspace*{-0.05in}

The paper provides a systematic study of local maximal monotonicity and its strong counterparts for set-valued mappings in infinite-dimensional spaces, while all the established characterizations of these properties are new even in finite dimension. The study of local maximal monotonicity, besides being of its own interest, is highly motivated by applications to important problems of optimization and operations research by using powerful methods of variational analysis. A large portion of our future research will aim at applications to recent concepts of variational and strong variational convexity in both finite and infinite dimensions that have drawn a great interest in various aspects of optimization theory and numerical algorithms; see, e.g., \cite{kkmp23,kmp22convex,rtr251,r22}.

\end{document}